\documentclass[12pt]{amsart}

\usepackage{amssymb, graphicx, url}

\newtheorem{thm}{Theorem}
\newtheorem{prop}[thm]{Proposition}

\newtheorem{cor}[thm]{Corollary}
\theoremstyle{definition}
\newtheorem{definition}[thm]{Definition}

\newtheorem{rem}[thm]{Remark}

\newcommand{\uC}{\mathbf{C}}

\title{Properties of the corolla polynomial of a 3-regular graph}
\author{Dirk Kreimer}
\address{Humboldt U., Depts. of Math and Phys, Unter den Linden 6, 10099 Berlin, Germany}
\author{Karen Yeats}
\address{Dept. of Math, Simon Fraser University, 8888 University Dr, Burnaby BC, Canada}
\thanks{Dirk Kreimer is supported by the Alexander von
Humboldt Foundation and the BMBF through an Alexander von Humboldt Professorship.  Karen Yeats is partially supported by an NSERC discovery grant and would like to thank Humboldt University for its hospitality and Spencer Bloch for asking about universal corolla polynomials.}

\begin{document}
\begin{abstract}
We investigate combinatorial properties of a graph polynomial indexed by half-edges of a graph which was introduced recently
to understand the connection between Feynman rules for scalar field theory and Feynman rules for gauge theory. We investigate the new graph polynomial as a stand-alone object.
\end{abstract}

\maketitle

\section{Introduction} Recently, one of us (DK) defined from a 3-regular graph $G$ a polynomial $C(G)$ in variables indexed by the half-edges of the graph.  $C(G)$ is manufactured to bootstrap gauge theory Feynman rules from scalar graphs. 
Replacing the half-edge variables in $C(G)$ by suitable differential operators makes $C(G)$ act as a differential operator $C^ D(G)$ on the scalar integrand 
\[
\frac{e^{-\frac{\phi(G)}{\psi(G)}}}{\psi^2(G)},
\]
where $\psi,\phi$ are the first and second Symanzik polynomials\footnote{The first Symanzik polynomial is also known as the Kirchhoff polynomial (or the dual Kirchhoff polynomial, depending on conventions.)}. As a consequence, from the sum of connected 3-regular graphs at a given loop order, we obtain the full gauge theory amplitude at that loop order though the action of $C^ D(G)$, where 4-valent vertices are generated through a residue operation reflecting graph homology, while internal ghost loops are generated from the structure of $C(G)$.

These applications to gauge theory and physics are worked out in a collaboration of DK with Matthias Sars and Walter van Suijlekom
\cite{KSvS}.

The purpose of this short note is to describe some combinatorial properties of $C(G)$ and generally make the argument that $C(G)$ is a nice object  in mathematics.

For the purposes of this paper graphs will be viewed as constructed out of half-edges.  Edges in the usual sense are pairs of half-edges, and will be known as \emph{internal edges}.  Unmatched half edges are also allowed and are called \emph{external edges}.  Multiple edges and loops in the sense of graph theory\footnote{tadpoles} are also allowed.  For all but the final section we are concerned with such graphs where each vertex is incident to exactly 3 half edges.  We will call these graphs 3-regular, but note that the external edges contribute to the valence, so only those with no external edges are 3-regular graphs in the usual sense.

Graphs in this sense are the correct object to describe the underlying structure of a Feynman diagram, ignoring the details of particle content, indices, etc.

Let $G$ be such a 3-regular graph.  We need the following definitions
\begin{definition}
\mbox{}

\begin{enumerate}
\item To a half-edge $j$ of $G$ associate the variable $a_j$.
\item For a vertex $v$ of $G$ let $n(v)$ be the set of half-edges incident to $v$.
\item For a vertex $v$ of $G$ let $D_v = \sum_{j \in n(v)} a_{j}$.
\item Let $\mathcal{C}$ be the set of all cycles\footnote{For the purposes of this paper cycle always means simple cycle, that is, no repeated vertices are allowed.} of $G$.
\item For $C$ a cycle and $v$ a vertex in $V$, since $G$ is 3-regular, there is a unique half-edge of $G$ incident to $v$ and not in $C$, let $o(C,v)$ be this half-edge.
\item For $i\geq 0$ let 
  \[
  C^i(G) = \sum_{\substack{C_1,C_2,\ldots C_i \in \mathcal{C} \\ C_j \text{pairwise disjoint}}} \left(\left(\prod_{j=1}^{i} \prod_{v \in C_j}a_{o(C_j,v)}\right)\prod_{v \not\in C_1\cup C_2\cup \cdots \cup C_i} D_v\right)
  \]
\item Let
  \[
   C(G) = \sum_{j \geq 0} (-1)^j C^j(G)
  \]
\end{enumerate}
\end{definition}

Note that $C(G)$ is a polynomial because $C^i(G)=0$ for $i>|\mathcal{C}|$.

\section{Properties}

The first nice property is that $C(G)$ counts something and hence has all nonnegative coefficients -- in fact all monomials appear with coefficient $0$ or $1$.

\begin{thm}\label{bin reformulation}
  Let $\mathcal{T}$ be the set of sets $T$ of half edges of $G$ with the property that
  \begin{itemize}
    \item every vertex of $G$ is incident to exactly one half edge of $T$
    \item $G\smallsetminus T$ has no cycles
  \end{itemize}
  Then
  \[
  C(G) = \sum_{T\in \mathcal{T}} \prod_{h\in T} a_h
  \]
\end{thm}

\begin{proof}
  First notice that every monomial in each $C^i(G)$ includes exactly one variable for each vertex of $G$.  This is because if a vertex $v$ is in one of the cycles, then it is in exactly one of the cycles as the cycles are disjoint, and so appears once as a $a_{o(C,v)}$, and if $v$ is not in one of the cycles, then $v$ contributes $D_v$.

  Consider a set $T$ of half edges of $G$ with the property that every vertex of $G$ is incident to exactly one half edge of $T$. Note that $G\smallsetminus T$ is 2-regular, and so consists of a disjoint union of cycles and lines;  lines are possible because of the external edges.  Let $k$ be the number of cycles of $G\smallsetminus T$. 

  Now we wish to count how many times $\prod_{h \in T}a_h$ appears in $C^j(G)$.   $\prod_{h \in T}a_h$ appears once for every set of cycles $C_1, C_2, \ldots C_j$ with the property that $a_{o(C_i,v)} \in T$ for all $v \in C_1, C_2\ldots C_j$, that is whenever $C_1\cup C_2 \cup \cdots \cup C_j \subseteq G\smallsetminus T$.  There are $\binom{k}{j}$ ways this can occur.

  Thus the number of times $\prod_{h \in T}a_h$ appears in $C(G)$ is  
  \[
     \sum_{\ell = 0}^{k} (-1)^\ell\binom{k}{\ell} \\
     = \begin{cases} (1-1)^k = 0 & k \neq 0 \\
    1 & k = 0\end{cases}
  \]
The result follows.
\end{proof}

\begin{rem}
For a graph $G$, let $E$ be a set of 
pairwise disjoint internal edges of $G$.
For $i\geq 0$ let 
  \[
  C^i_E(G) = \sum_{\substack{C_1,C_2,\ldots C_i \in \mathcal{C} \\ C_j \text{ pairwise disjoint} \\ C_j\cap E=\emptyset}} \left(\left(\prod_{j=1}^{i} \prod_{v \in C_j}a_{o(C_j,v)}\right)\prod_{v \not\in C_1\cup C_2\cup \cdots \cup C_i\cup E} 
  D_v\right)
  \]
  where the sum forbids cycles from sharing either vertices or edges with $E$.
\item Let
  \[
   C_E(G) = \sum_{j \geq 0} (-1)^j C^j_E(G).
  \]
Then,
\[
C_E(G)=C(G- E)
\]  
where $G-E$ is the graph with the edges and vertices involved in $E$ removed.  Removing a vertex removes all its indicent half-edges so that $2|E|$ new external edges are generated. 
Note that $C_\emptyset(G)=C(G)$.
\end{rem}

As a direct consequence of Theorem \ref{bin reformulation} we have the following corollary
\begin{cor}\label{cor mult}
  Let $G$ be the disjoint union of $G_1$ and $G_2$, then
  \[
    C(G) = C(G_1)C(G_2)
  \]
\end{cor}

\begin{definition}\label{c def}
  For $T$ a set of half edges of $G$ let 
   $c(T)$ be the sum of 
  \begin{itemize}
    \item the number of connected components of $G \smallsetminus T$
    \item the number of external edges in $T$
    \item the number of internal edges with for which both half edges are in $T$
  \end{itemize}
 \end{definition}

$c(T)$ is the correct notion of number of components for $G\smallsetminus T$.  As one simple example, using Euler's formula and 3-regularity one can easily compute that for $G$ connected and $T \in \mathcal{T}$
\[
c(T) = \frac{v+e_{\text{ext}}}{2}
\]
where $v$ is the number of vertices of $G$ and $e_{\text{ext}}$ is the number of external edges of $G$.  $c(T)$ is also useful for the universal corolla polynomial of the next section.

%
%
%

The second nice property of $C(G)$ is that there is a nice formula for how $C(G)$ decomposes upon removing a vertex.  Note that when we remove a vertex in a graph then we remove the half edges incident to the vertex, but we leave the other half of any affected edges.  These remaining halves are now external edges.

\begin{prop}\label{recurrence}
  Let $v$ be a vertex of $G$ which is not incident to an external edge of $G$ nor incident to both ends of an internal edge.  Label the half edges incident to $v$ by $1$, $2$, and $3$.  Let $H_i$ for $i=1,2,3$ be $G$ with $v$ removed and the two half edges which were not paired with half edge $i$ joined to make a new internal edge.  Then
  \[
    C(G) = a_{1}C(H_1) + a_{2}C(H_2) + a_{3}C(H_3)
  \]
  In pictures
  \[
  C\left(\raisebox{-.7cm}{\includegraphics[scale=0.7]{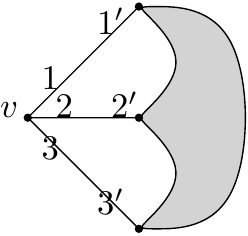}}\right) = a_{1}C\left(\raisebox{-.7cm}{\includegraphics[scale=0.7]{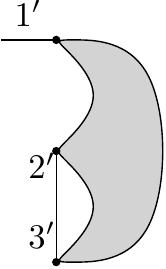}}\right) + a_{2}C\left(\raisebox{-.7cm}{\includegraphics[scale=0.7]{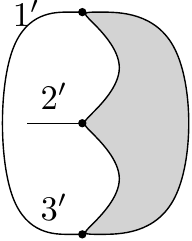}}\right) + a_{3}C\left(\raisebox{-.7cm}{\includegraphics[scale=0.7]{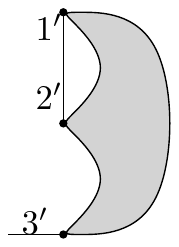}}\right)
  \]
\end{prop}

Note that $G\smallsetminus v$ does not need to be connected.

\begin{proof}
  Let $i'$ label the other half edge paired with half edge $i$ for $i=1,2,3$.

  Collect together those terms of $C(G)$ where the half edge $1$ is removed.  These are exactly the ways of removing a half edge from each remaining vertex so as to result in no cycles of $G\smallsetminus v$ and so that joining the remaining $2'$ and $3'$ does not cause a cycle.  These are exactly the terms of $C(H_1)$.

The argument runs likewise for $i=2$ and $i=3$.
\end{proof}

The previous proposition gives the basic recurrence from which we can build and understand $C(G)$.  Its role is comparable to a contraction-deletion formula, but it is vertex based since we are working with corolla polynomials.  We still need to characterize the recurrence when external edges or loops in the sense of graph theory are involved, and we need to give the base case.  These are listed in picture form in the following proposition.  The proofs are as above and hence are not given.

\begin{prop}\label{other cases}
\allowdisplaybreaks
  \begin{align*}
  C\left(\raisebox{-.4cm}{\includegraphics[scale=0.7]{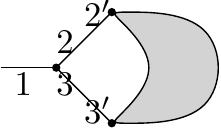}}\right) & = a_{1}C\left(\raisebox{-.4cm}{\includegraphics[scale=0.7]{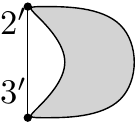}}\right) + (a_{2}+a_{3})C\left(\raisebox{-.4cm}{\includegraphics[scale=0.7]{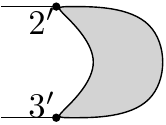}}\right) \\
  C\left(\raisebox{-.4cm}{\includegraphics[scale=0.7]{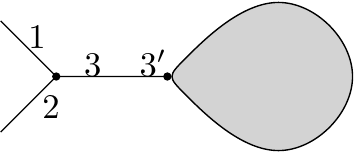}}\right) & = (a_1+a_2+a_3)C\left(\raisebox{-.4cm}{\includegraphics[scale=0.7]{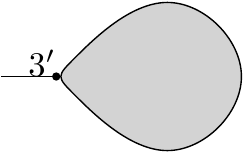}}\right) \\
  C\left(\raisebox{-.4cm}{\includegraphics[scale=0.7]{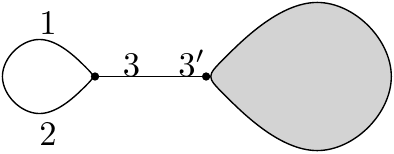}}\right) & = (a_1+a_2)C\left(\raisebox{-.4cm}{\includegraphics[scale=0.7]{H14}}\right) \\
  C\left(\raisebox{-.4cm}{\includegraphics[scale=0.7]{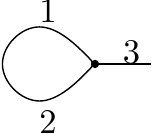}}\right) & = a_1+a_2 \\
  C\left(\raisebox{-.4cm}{\includegraphics[scale=0.7]{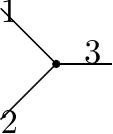}}\right) & = a_1+a_2+a_3
  \end{align*}
\end{prop}

In view of the generalization in the next section, we have chosen to give the cases in Propositions \ref{recurrence} and \ref{other cases} separately rather than unify them in view of the generalization in the next section.

\section{The universal 3-regular corolla polynomial}

A natural question now is to ask what is the most general graph polynomial satisfying Proposition \ref{recurrence}.  Of course the answer to such a question depends on our assumptions of what sorts of graph polynomials could count.  One possible answer is the following.

\begin{definition}
  Let $\mathcal{H}$ be the set of sets of half edges of $G$ containing exactly one half edge incident to each vertex.  Let $\mathbf{a} = (a_1,a_2,\ldots)$ be indeterminants, one for each half edge, and let $r$ be another indeterminant.  Define
  \[
  \uC(G,r,\mathbf{a}) = \sum_{H\in \mathcal{H}} r^{\ell(G\smallsetminus H)}\prod_{h\in H}a_h
  \]
  where $\ell(G\smallsetminus H)$ is the number of independent cycles of $G\smallsetminus H$.
\end{definition}



$\uC(G,r,\mathbf{a})$ is universal in the following sense.

\begin{thm}
  Let $R$ be a ring.  For any 3-regular graph $G$ (with external edges allowed) let $\mathbf{a}_G$ be indeterminants indexed by the half edges of $G$.  Let $f$ be a function on 3-regular graphs with $f(G) \in R[\mathbf{a}_G]$ which has the following properties
  \begin{itemize}
    \item For any vertex $v$ of $G$, $f(G)$ is homogeneous of degree 1 in the half edges incident to $v$.
    \item $f$ satisfies Proposition \ref{recurrence} and Corollary \ref{cor mult}.
  \end{itemize}
Then for any $G$,
\[
  f(G) = r_0(G)\uC(G,r_1,\mathbf{a}_G)
\]
for some $r_1 \in R$ independent of $G$ and some $r_0(G)$ in the field of fractions of $R$.  
\end{thm}

\begin{proof}
  For the proof it will be convenient to work with the following variant of $\uC(G,r,\mathbf{a})$.  Let
  \[
  \widetilde{\uC}(G,q,r,\mathbf{a}) = \sum_{H\in \mathcal{H}} q^{c(H)}r^{\ell(G\smallsetminus H)}\prod_{h\in H}a_h
  \]
  where $c(H)$ is as defined in Definition \ref{c def}.
  Suppose we remove the half edges in $H$ from $G$ one by one; each removal either decreases the number of independent cycles by one, creates a new component, completes the removal of an internal edge of $G$, or removes an external edge of $G$.  Thus
  \[
    (\ell(G)-\ell(G\smallsetminus H)) + c(H) - c = |H| = v
  \]
  where $v$ is the number of vertices of $G$ and $c$ is the number of connected components of $G$.  Solving,
  $
  c(H) = v - \ell(G) + c + \ell(G\smallsetminus H)
  $
  and so
  \[
  \widetilde{\uC}(G,q,r,\mathbf{a}) = q^{v-\ell(G)+c}\sum_{H\in \mathcal{H}} (qr)^{\ell(G\smallsetminus H)}\prod_{h\in H}a_h = q^{v-\ell(G)+c}\uC(G,qr,\mathbf{a})
  \]

  Next note that if we understand $f$ on each of the left hand sides from Proposition \ref{other cases} then we know it completely.  We will consider each of these in turn, comparing to $\widetilde{\uC}$.
 
  Calculating from the definition we obtain
  \[
    \widetilde{\uC}\left(\raisebox{-.4cm}{\includegraphics[scale=0.7]{G1}}\right) = qa_{1}\widetilde{\uC}\left(\raisebox{-.4cm}{\includegraphics[scale=0.7]{H1}}\right) + (a_{2}+a_{3})\widetilde{\uC}\left(\raisebox{-.4cm}{\includegraphics[scale=0.7]{Gmv1}}\right)
    \]
    Now consider applying $f$ to 
    \[
      K_1 = \raisebox{-.4cm}{\includegraphics[scale=0.7]{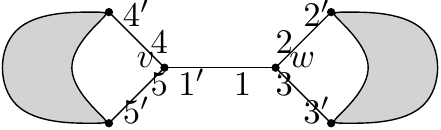}}.
    \]
    We can compute this in two possible ways, first applying the recurrence from Proposition \ref{recurrence} to $v$ and then to $w$, or first to $w$ and then to $v$, in both cases simplifying with the multiplicative property when appropriate.  Reducing first $v$ then $w$ we get
    \begin{align*}
    f(K_1) & = a_{1'}f\left(\raisebox{-.4cm}{\includegraphics[scale=0.7]{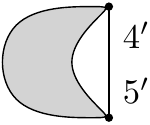}}\right)f\left(\raisebox{-.4cm}{\includegraphics[scale=0.7]{G1}}\right) + (a_4+a_5)a_1f\left(\raisebox{-.4cm}{\includegraphics[scale=0.7]{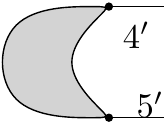}}\right)f\left(\raisebox{-.4cm}{\includegraphics[scale=0.7]{H1}}\right) \\
    & + a_5a_3f\left(\raisebox{-.4cm}{\includegraphics[scale=0.7]{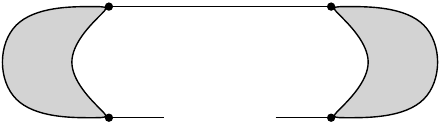}}\right) + a_5a_2f\left(\raisebox{-.4cm}{\includegraphics[scale=0.7]{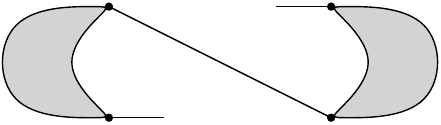}}\right) \\
    & + a_4a_3f\left(\raisebox{-.4cm}{\includegraphics[scale=0.7]{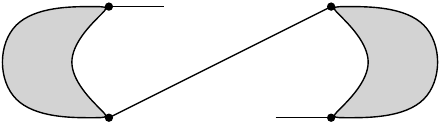}}\right) + a_4a_2f\left(\raisebox{-.4cm}{\includegraphics[scale=0.7]{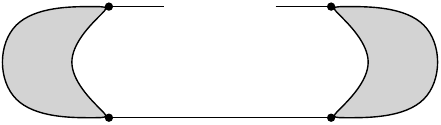}}\right)
    \end{align*}
    Reducing $w$ first and then $v$ we get
        \begin{align*}
    f(K_1) & = a_{1}f\left(\raisebox{-.4cm}{\includegraphics[scale=0.7]{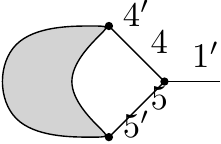}}\right)f\left(\raisebox{-.4cm}{\includegraphics[scale=0.7]{H1}}\right) + (a_2+a_3)a_1'f\left(\raisebox{-.4cm}{\includegraphics[scale=0.7]{H1rev}}\right)f\left(\raisebox{-.4cm}{\includegraphics[scale=0.7]{Gmv1}}\right) \\
    & + a_5a_3f\left(\raisebox{-.4cm}{\includegraphics[scale=0.7]{K21}}\right) + a_5a_2f\left(\raisebox{-.4cm}{\includegraphics[scale=0.7]{K31}}\right) \\
    & + a_4a_3f\left(\raisebox{-.4cm}{\includegraphics[scale=0.7]{K41}}\right) + a_4a_2f\left(\raisebox{-.4cm}{\includegraphics[scale=0.7]{K51}}\right)
    \end{align*}
        Let
        \begin{align*}
          H_1 &  = \raisebox{-.4cm}{\includegraphics[scale=0.7]{H1}} \quad 
          H'_1 = \raisebox{-.4cm}{\includegraphics[scale=0.7]{H1rev}}  \quad 
          J_1  = \raisebox{-.4cm}{\includegraphics[scale=0.7]{Gmv1}} \\
          J_1' & = \raisebox{-.4cm}{\includegraphics[scale=0.7]{Gmv1rev}} \quad
          G_1  = \raisebox{-.4cm}{\includegraphics[scale=0.7]{G1}} \quad
          G_1'  = \raisebox{-.4cm}{\includegraphics[scale=0.7]{G1rev}}
        \end{align*}
        Equating the two expressions for $f(K_1)$ we see that
        \begin{align*}
          & a_{1'}f(H_1')f(G_1) + (a_4+a_5)a_1f(J_1')f(H_1) \\
          &  =  a_{1}f(G_1')f(H_1) + (a_2+a_3)a_1'f(H_1')f(J_1)
        \end{align*}
        so the coefficient of $(a_2+a_3)$ in $f(G_1)$ must be $f(J_1)$ and the coefficient of $a_1$ in $f(G_1)$ must be some multiple of $f(H_1)$.  Moreover this multiple must be the same as for the coefficient of $a_{1'}$ in $f(G_1')$ and hence must be independent of the graph.  That is, there is some $r_2 \in R$ such that
        \[
        f\left(\raisebox{-.4cm}{\includegraphics[scale=0.7]{G1}}\right) =  r_2a_{1}f\left(\raisebox{-.4cm}{\includegraphics[scale=0.7]{H1}}\right) + (a_2+a_3)f\left(\raisebox{-.4cm}{\includegraphics[scale=0.7]{Gmv1}}\right) 
        \]

        Proceeding more tersely in the remaining cases we have
        \[
  \widetilde{\uC}\left(\raisebox{-.4cm}{\includegraphics[scale=0.7]{G4}}\right)  = q(a_1+a_2+a_3)\widetilde{\uC}\left(\raisebox{-.4cm}{\includegraphics[scale=0.7]{H14}}\right) \\
        \]
        Applying $f$ to
        \[
        \includegraphics[scale=0.7]{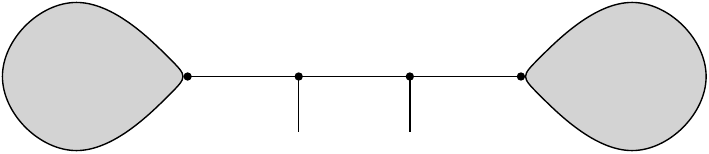}
        \]
        in the two different ways we obtain
        \[
         f\left(\raisebox{-.4cm}{\includegraphics[scale=0.7]{G4}}\right)  = r_2(a_1+a_2+a_3)\widetilde{\uC}\left(\raisebox{-.4cm}{\includegraphics[scale=0.7]{H14}}\right) \\
        \]

        We also have
        \[
          \widetilde{\uC}\left(\raisebox{-.4cm}{\includegraphics[scale=0.7]{G5}}\right)  = (a_1+a_2)\widetilde{\uC}\left(\raisebox{-.4cm}{\includegraphics[scale=0.7]{H14}}\right) + qra_3\widetilde{\uC}\left(\raisebox{-.4cm}{\includegraphics[scale=0.7]{H14}}\right)
        \]
        and applying $f$ to
        \[
        \includegraphics[scale=0.7]{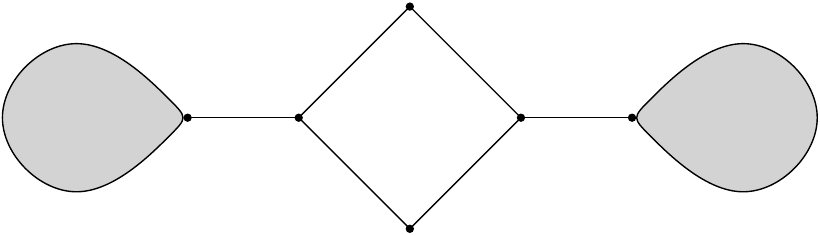}
        \]
        in both ways gives that there is an element $r_3\in R$ independent of the graph so that
        \[
        f\left(\raisebox{-.4cm}{\includegraphics[scale=0.7]{G5}}\right) = (a_1+a_2)f\left(\raisebox{-.4cm}{\includegraphics[scale=0.7]{H14}}\right) + r_3a_3f\left(\raisebox{-.4cm}{\includegraphics[scale=0.7]{H14}}\right)
        \]

        Finally,
        \[
         \widetilde{\uC}\left(\raisebox{-.4cm}{\includegraphics[scale=0.7]{G6}}\right) = q(a_1+a_2) + q^2ra_3 \qquad
  \widetilde{\uC}\left(\raisebox{-.4cm}{\includegraphics[scale=0.7]{G7}}\right) = q^2(a_1+a_2+a_3)
  \]
  and $f$ applied to 
  \[
  \includegraphics[scale=0.7]{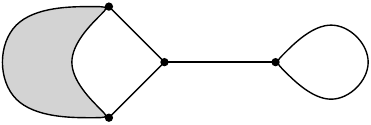} \quad \text{and} \quad \includegraphics[scale=0.7]{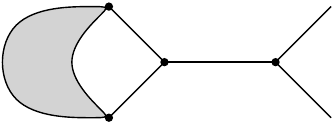}
  \]
  give
  \[
         f\left(\raisebox{-.4cm}{\includegraphics[scale=0.7]{G6}}\right) = r_2(a_1+a_2) + r_2r_3a_3 \qquad
  f\left(\raisebox{-.4cm}{\includegraphics[scale=0.7]{G7}}\right) = r_2^2(a_1+a_2+a_3)
  \]

  Temporarily move to any larger ring containing $R$ and $r_2^{-1}$.  Taking these calculations all together we have that
  \[
  f(G) = \widetilde{\uC}(G,r_2, r_3/r_2, \mathbf{a}) = r_2^{v-\ell(G)+c}\uC(G,r_3,\mathbf{a})
  \]
  which is what we wanted to prove.
\end{proof}

The structure of this universal corolla polynomial is reminiscent of that of the multivariate Tutte polynomial \cite{Smvtutte}.

Consider special evaluations of $\uC$.
{}From Theorem \ref{bin reformulation} we have 
\[
C(G) = \uC(G,0,\mathbf{a}).
\]  
Some other special evaluations are trivial; setting $r=1$ forgets the cycles and so simply sums over all sets of half edges which are homogeneous linear at each vertex, that is
\[
\uC(G,1,\mathbf{a}) = \prod_{v \in V(G)}D_v
\]
while setting all $a_i=1$ gives the generating function for $\mathcal{H}$ counting by the number of cycles of $G\smallsetminus H$ for $H \in \mathcal{H}$.  By homogeneity setting all $a_i$ to any constant simply scales the preceding polynomial.  A more interesting possibility is to impose $\sum_{h\in n(v)} a_h =0$ for each vertex $v$.
  
\begin{prop}
  Let $\widetilde{\mathbf{a}}$ be an assignment of edge labels to $G$ with the restriction that $\sum_{h\in n(v)} a_h =0$  for each vertex $v$.  Then
  \[
    \uC(G,r,\widetilde{\mathbf{a}}) = \sum_{k \geq 0}\sum_{\substack{C_1,\ldots,C_k\\ \text{disjoint cycles} \\\text{spanning }G}}(r-1)^k\prod_{j=1}^{k}\prod_{v\in C_j} a_{o(C_j,v)}
  \]
\end{prop}

\begin{proof}
If $r$ is a positive integer then for arbitrary edge variables we can rewrite $\uC$ as follows
\begin{equation}\label{potts type eq}
\uC(G,r,\mathbf{a}) = \sum_{f:\mathcal{C} \rightarrow \{1,\ldots r\}} \prod_{v\in V(G)}\sum_{h\in n(v)} a_h\epsilon(h,f)
\end{equation}
where 
\[
\epsilon(h,f) = 
\begin{cases}
  1 & \forall C\in \mathcal{C}, h\in C \Rightarrow f(C)=1\\
  0 & \text{otherwise}
\end{cases}
\]
To see that this holds take $H\in \mathcal{H}$.  Consider all colourings of the set of cycles with $r$ colours, choose a distinguished colour (colour 1), and require all cycles going through edges of $H$ to be of colour 1.  The cycles not in $H$ then are free to run over any colour and hence each contribute one power of $r$.  This is loosely a Potts-model-type formulation of $\uC$.

Returning to $\widetilde{\mathbf{a}}$.  Choose an $f:\mathcal{C} \rightarrow \{1,\ldots,r\}$.  $f$ along with a vertex $v$ will contribute $0$ to \eqref{potts type eq} if either all $\epsilon(h,f)$ are $0$ or all are $1$.  All are $0$ when at least one cycle not of colour 1 goes through each half edge incident to $v$.  All are $1$ when no cycles of colour 1 go through $v$.  Provided none of the $\widetilde{\mathbf{a}}$ are $0$ all other terms are nonzero.  

Thus terms which are not $0$ are indexed by the colourings of cycles where at each vertex exactly two half edges are involved in one or more cycles of colour $>1$.  If two such cycles passed through any vertex then at the first vertex where they differ all half edges would be involved in a cycle of colour $>1$, so in fact we must have exactly one cycle of colour $>1$ at each vertex.  Thus the nonzero terms of $\uC(G,r,\widetilde{\mathbf{a}})$ are indexed by disjoint sets of cycles which span $G$, and the corresponding monomial in the $\widetilde{\mathbf{a}}$ is the product of half edge variables not in the set of cycles.  Each such set of cycles contributes once for each possibly colouring of the cycles from colours $2,\ldots, r$, that is the set of cycles contributes $(r-1)^k$ times.

This gives the desired result for positive integer $r$.  For any fixed values of $\widetilde{\mathbf{a}}$, $\uC(G,r,\widetilde{\mathbf{a}})$ is a polynomial in $r$ and hence the result holds for all $r$. 
\end{proof}

\section{Higher valences}

A natural question which remains is what the story is for graphs with higher valences.  Physics doesn't need this because gauge theories only use 3 and 4 valent vertices, and the 4 valent vertices are created automatically from the differential operators built from $C(G)$, see \cite{KSvS}.  None the less there is a natural choice.

\begin{itemize}
\item Let $G$ be a graph with no degree restrictions and potentially with external edges.  For a vertex $v\in V(G)$ let 
  \[
  E_v = \sum_{\substack{j,k \in n(v) \\ j \neq k}} \prod_{\substack{i \neq j \\ i\neq k}}a_{i}.
  \]
\item For $C$ a cycle and $v$ a vertex, let 
  \[
  p(C,v) = \prod_{\substack{h \in n(v) \\ h \not\in C}}a_{h}
  \]
\item For $i\geq 0$ let 
  \[
  C^i = \sum_{\substack{C_1,C_2,\ldots C_i \in \mathcal{C} \\ C_j \text{pairwise disjoint}}} \left(\left(\prod_{j=1}^{i} \prod_{v \in C_j}p(C,v)\right)\prod_{v \not\in C_1\cup C_2\cup \cdots \cup C_i} E_v\right)
  \]
\item Let
  \[
   C_{k} = \sum_{j \geq 0} (-1)^j C_{k}^j
  \]
\end{itemize}

In this definition we are summing over sets $T$ of half edges for which all but 2 of the half edges adjacent to any vertex $v$ of $G$ are in $T$.  As a consequence $G\smallsetminus T$ is again a disjoint union of cycles and lines, so the cancellation argument, Theorem \ref{bin reformulation}, goes through unchanged.  There is also something resembling contraction-deletion,

\begin{prop}
  Let $1$ be an internal edge of $G$ with half edges $h$ and $k$. Then
  \[
    C(G) = C(G/ 1) + a_{h}C(G\smallsetminus h) + a_{k}C(G\smallsetminus k) - a_{h}a_{k}C(G\smallsetminus hk)
  \]
\end{prop}

\begin{proof}
  Consider a set $T$ of half edges of $G$ for which all but two of the half edges adjacent to any vertex $v$ of $G$ are in $T$ and $G\smallsetminus T$ has no cycles.
  
  If neither $h$ nor $k$ are in $T$, then edge 1 can be contracted without changing the cycle structure of $G\smallsetminus T$ and the new vertex still has exactly two edges in $T$.

  If $h$ is not in $T$ but $k$ is, then edge 1 can be cut without changing the cycle structure of $G\smallsetminus T$.  To preserve all but two edges in $T$ at every vertex, we must not remove $h$ from $G$.  Thus this term appears in $a_kC(G\smallsetminus k)$. 

  However, $a_kC(G\smallsetminus k)$ also contains the terms where $h$ happens to be in the set of half edges.  That is, it contains the terms of $C(G)$ where both $h$ and $k$ are in $T$.  Such terms are exactly the terms of $a_ha_kC(G\smallsetminus hk)$.

  The same argument applies with $a$ and $b$ swapped.  So we get that all terms of $C(G)$ appear in $C(G/ 1) + a_hC(G\smallsetminus h) + a_kC(G\smallsetminus k)$ with those in $a_ha_kC(G\smallsetminus hk)$ appearing twice and the others appearing once.  Subtracting off the over counting gives the result.
\end{proof}

\bibliographystyle{plain}
\bibliography{main}

\end{document}